\documentclass[a4paper,12pt]{article}
\usepackage[ansinew]{inputenc}
\usepackage{amsfonts}
\usepackage{latexsym}
\usepackage{amsmath}
\usepackage{amssymb}
\usepackage{multirow}
\usepackage{eepic}
\usepackage{graphicx}
\usepackage{color}
\usepackage{tikz}
\usepackage{geometry}

\newcommand{\C}{\mathbb{C}}
\newcommand{\N}{\mathbb{N}}

\newtheorem{defin}{Definition}[section]

\newtheorem{theorem}[defin]{Theorem}
\newtheorem{exa}{Example}[section]
\newenvironment{example}{\begin{exa}\rm}{\end{exa}}
\newtheorem{exas}{Examples}

\newtheorem{lemma}[defin]{Lemma}

\newtheorem{corollary}[defin]{Corollary}
\newenvironment{proof}
{\noindent{\it Proof.}}{\hfill $\Box$\par\vspace{2.5mm}}

\numberwithin{equation}{section}

\title{$Q$-difference analogue of the Stothers-Mason theorem}
\vspace{0.5cm}

\author{Jian-Tang Lu, Xing-Xing Lu\footnote{X.-X. Lu is the corresponding author} and Zhi-Tao Wen\footnote{\ Z.-T. Wen is supposed by the National Natural Science Foundation of China (No.~12471076)}
 }
\date{}

\begin{document}
\maketitle

\begin{abstract}

In this paper, we give a new definition of the $q$-weight of zeros,  which reduces to the multiplicity of zeros as $q\to 1$. Furthermore, we obtain a $q$-difference version of the Stothers-Mason theorem by means of the new definition of the $q$-difference radical, which covers the classical Stothers-Mason theorem as $q\to 1$. As applications, we study the polynomial solutions of $q$-difference Fermat type functional equations.

\medskip
\noindent
\textbf{Keyword}: $q$-weight of zeros, $q$-difference radical, Stothers--Mason theorem, Fermat type functional equation

\medskip
\noindent
\textbf{2020MSC}: 39B32, 30D35

\end{abstract}

\section{Introduction}
Let $P$ be a polynomial. The radical ${\rm rad}(P)$ is the product of distinct linear factors of $P$. Let $a$, $b$ and $c$ be relatively prime polynomials such that not all of them are identically zero. The Stothers--Mason theorem~\cite{mason1984},~\cite{stothers1981}, see also~\cite{Snyder}, states that if they satisfy $a+b=c$, then
    $$
    \max\{\deg(a),\deg(b),\deg(c)\}\leq \deg({\rm rad}(abc))-1.
    $$

A difference analogue of the Stothers-Mason theorem or difference $abc$ theorem for polynomials was given by Ishizaki et al. in \cite[Theorem~3.1]{IKLT}. Let $a,b$ and $c$ be relatively prime polynomials in $\C[z]$, which are not all constants, satisfying $a+b=c$. Then,
    $$
    \max\{\deg(a),\deg(b),\deg(c)\}\leq \tilde{n}_\kappa(a)+\tilde{n}_\kappa(b)+\tilde{n}_\kappa(c)-1,
    $$
where $\kappa\in\C\setminus\{0\}$, and
    $$
    \tilde{n}_\kappa(p)=\sum_{w\in\C}(\text{ord}_w(p)-\min\{\text{ord}_w(p),\text{ord}_{w+\kappa}(p)\})
    $$
for a polynomial $p$ by $\text{ord}_w(p)$ denoting the order of zeros of $p$ at $z=w$.

Another difference analogue of the Stothers-Mason theorem was given by Ishizaki and Wen in \cite[Theorem~3.5]{IW2022}. Let $a, b$ and $c$ be relatively shifting prime polynomials, see \cite[pages~736-737]{IW2022}, such that $a, b$ and $c$ are not all constants satisfying $a+b=c$. Then
    $$
    \max\{\deg(a),\deg(b),\deg(c)\}\leq \deg(\text{rad}_\Delta(abc))-1,
    $$
where $\text{rad}_\Delta(abc)$ is difference radical of $abc$ defined in \cite{IW2022}.

One of the purposes of this paper is to introduce a $q$-difference analogue of the radical of polynomials and a $q$-difference analogue of the Stothers-Mason theorem, which are different from the results in \cite{DCL}. These $q$-difference counterparts reduce to the classical radical of polynomials and Stothers-Mason theorem, respectively, as $q\to 1$. As applications, we study polynomial solutions of $q$-difference Fermat type functional equations.

In Section 2, we define the $q$-weight of a zero of an analytic function in a domain $G$, and investigate its properties
via $q$-difference calculus. We then introduce the $q$-difference radical by means of the $q$-weight of zeros and focus on the $q$-difference analogue of the Stothers-Mason theorem by using the
$q$-difference radical in Section 3. Section 4 is concerned with an extension of the aforementioned $q$-difference Stothers-Mason theorem. Finally, in Section 5 we study the polynomial
solutions of $q$-difference version of the Fermat type functional equations.

\section{$Q$-weight of zeros}\label{dirad}

Following \cite{KC}, we denote $q$-number which is the $q$-analogue of the natural
number $n$ by
    $$
    [n]_q=\frac{q^n-1}{q-1}=q^{n-1}+\cdots+q+1.
    $$
We set the $q$-analogue of factorial $n!$ by
    $$
    [n]_q!=
    \begin{cases}
    1 \quad &n=0;\\
    [n]_q\times[n-1]_q\times \cdots \times [1]_q \quad &n\in\N^+.
    \end{cases}
    $$
Let us set the $q$-analogue of power function $(z-a)^n$ by
    $$
    [z-a]^n_q=
    \begin{cases}
    1 \quad &n=0;\\
    (z-a)(z-aq)\cdots(z-aq^{n-1}) \quad &n\in\N^+.
    \end{cases}
    $$
We note that if $a=0$, then $[z]^n_q=z^n$ for all $n\in\N^+\cup\{0\}$. For a given function $f$, we define $q$-derivative of $f$ \cite{Jackson1909} by
    $$
    D_q f(z)=\frac{f(qz)-f(z)}{qz-z}
    $$
for $z\in\C\setminus\{0\}$ and $q\neq 1$. In addition, $q$-derivative of $f$ at zero is defined by
    $$
    D_q f(0)=\lim_{n\to\infty}\frac{f(z'q^n)-f(0)}{z'q^n}
    $$
for $z'\in\C\setminus\{0\}$. Note that if $f(z)$ is analytic at $z=0$, then $D_qf(0)=f'(0)$.
Moreover, if $f$ is an analytic function, then $\lim_{q\to 1}D_qf=f'$, and $D_q$ is known as the \emph{Jackson operator}, see e.g.,\cite[Page~41]{Jackson1910}. For any $n\in\N^+$, set $D_q^nf=D_q^{n-1}(D_qf)$ and $D_q^0f=f$.
Obviously, for all $n\in\N^+$ we have
    $$
    D_q([z-a]_q^n)=[n]_q[z-a]^{n-1}_q,
    $$
which is $q$-analogue of differential calculus $((z-a)^n)'=n(z-a)^{n-1}$.
It is a fact that $D_q$ reduces the degree of a polynomial by one.

Let $f$ be analytic in a domain $G\subset\C$ and let $n\in\N$ and $q\in\C\setminus\{0\}$.
Suppose that $z_0, qz_0,\ldots, q^nz_0\in G$.
The point $z_0$ is called a \emph{zero of $f(z)$ with $q$-weight $n$ in $G$} provided $f(z_0)$ and all $q$-derivative $D_q^k f(z_0)$ vanish for all $0\leq k<n$, but $D_q^n f(z_0)\neq 0$. In general, the point $z_0$ is called an \emph{$a$-point of $f(z)$ with $q$-weight $n$ in $G$} provided $f(z_0)-a$ and all $q$-derivative $D_q^k f(z_0)$ vanish for all $0\leq k<n$, but $D_q^n f(z_0)\neq 0$.
In particular, if $z_0$ is an $a$-point of $f$ with $q$-weight 1, then we also call $z_0$ is an $a$-point of $f$ with simple $q$-weight. Obviously, $z_0=0$ being a zero of $f$ with $q$-weight $n$ is equivalent to $z_0=0$ being a zero of $f$ with multiplicity $n$.

We emphasize here that $q$-weight of zeros is a generalization of multiplicity of zeros of an analytic function. As $q\to 1$, we see that $q$-weight of zeros reduces to the multiplicity of zeros of an analytic function.

\begin{example}
Suppose that $n\in\N^+$ and $a, q\in\C$ such that $q\neq 0$, and $f(z)=[z-a]_q^n$. Then we have $D_q^kf(a)=0$ for any $k\in\N$ satisfying $0\leq k<n$, but $D_q^nf(a)\neq 0$. Thus, it is obvious that $z=a$ is a zero of $f$ with $q$-weight $n$. In addition, we see that $z=a$ becomes a zero of $f$ with multiplicity $n$ as $q\to 1$.
\end{example}

The following example  illustrates that a transcendental function may have a zero with infinite $q$-weight.

\begin{example}
Suppose that $0<|q|<1$, and a function
    $$
    f(z)=\prod_{n=0}^\infty(1+q^nz),
    $$
which converges uniformly on any compact subset of $\{z\in\C: |z|<R\}$. The entire function $f(z)$ is a solution of linear $q$-difference equations
    $$
    f(qz)-(1+qz)f(z)=0.
    $$
In addition, $f$ is of logarithmic order $\rho_{\log}(f)=2$, see \cite[Theorem~1.1]{Wen2014}.
 The logarithmic order is defined by
    $$
    \rho_{\log}(f):=\limsup_{r\to\infty}\frac{\log\log M(r,f)}{\log\log r}.
    $$
Obviously, the distinct zero sequence of $f$ is $z_n=-(1/q)^n$ for $n=0,1,\ldots$. Thus, $z_0=-1$ is a zero of $f$ with infinite $q$-weight.
\end{example}

It is well known that $z=z_0$ is a zero of an analytic function $f$ with multiplicity $n$ if
and only if $f$ is of the form $f(z)=(z-z_0)^ng(z)$, where $g(z)$ is analytic at $z = z_0$ and
$g(z_0)\neq 0$. The $q$-analogue of this result is given as follows.

\begin{theorem}\label{zero.theorem}
Let $f$ be analytic in the domain $G\subset\C$ and let $n\in\N$ and $q\in\C\setminus\{0\}$. Suppose that $z_0, qz_0,\ldots, q^nz_0\in G$.
The point $z_0$ is a zero of $f(z)$ with $q$-weight $n$ in $G$ if and only if there exists an analytic function $g(z)$ in $G$ such that
    \begin{equation}\label{fzero.eq}
    f(z)=[z-z_0]_q^ng(z)
    \end{equation}
and $g(q^nz_0)\neq 0$.
\end{theorem}

Before the proof of Theorem \ref{zero.theorem}, let us recall a lemma, see e.g.,\cite[Pages~12-13]{Bangerezako}, which reveals the relation between $f(q^nz)$ and $D_q^nf(z)$ for $n\in\N$.

\begin{lemma}\label{lm2}
	For any $z \in \mathbb{C}$, $k \in \mathbb{N}$, we have
    $$
	f(q^kz)=\sum_{j=0}^{k} q^{j(j-1)/2}(qz-z)^j
	\left[\begin{array}{cc}
	k\\j
	\end{array}\right]_q
	D_q^jf(z),
    $$
where
    $$
    \left[\begin{array}{cc}
    k\\
    j
    \end{array}
    \right]_q
    =\frac{[k]_q!}{[j]_q![k-j]_q!}
    $$
are called $q$-binomial coefficients. In addition, for $z \in\C\setminus\{0\}$ and $k\in\N$, we have
    $$
    D_q^kf(z) = \frac{1}{q^{k(k-1)/2}(qz-z)^k}\sum_{j=0}^k (-1)^{k-j}q^{(k-j)(k-j-1)/2}
    \left[\begin{array}{cc}
    k\\j
    \end{array}\right]_q
    f(q^jz).
    $$
\end{lemma}

\bigskip
\noindent
\emph{Proof of Theorem \ref{zero.theorem}}.
It is easy to get our assertion when $q=1$. In the following, we assume that $q\neq 1$.
Let $z_0\neq 0$ be a zero of $f$ with $q$-weight $n$ in a domain $G$, and $k\in\N$. Since $D_q^k f(z_0)$ vanish for $0\leq k<n$, by using Lemma~\ref{lm2}, we have
   $$ f(q^kz_0)=\sum_{j=0}^{k} q^{j(j-1)/2}(qz_0-z_0)^j
   \left[\begin{array}{cc}
   k\\j
   \end{array}\right]_q
   D_q^jf(z_0)=0.$$
In addition, we found that $f(q^nz_0)=q^{n(n-1)/2}D_q^nf(z_0)(qz_0-z_0)^n\neq 0$.
Hence, $f$ is of the form \eqref{fzero.eq}.
If $f$ is of the form \eqref{fzero.eq}, we see that $f(q^kz_0)=0$ for all $0\leq k<n$, but
$f(q^nz_0)\neq 0$. From Lemma \ref{lm2}, we have
$$
 D_q^kf(z_0) = \frac{1}{q^{k(k-1)/2}(qz_0-z_0)^k}\sum_{j=0}^k (-1)^{k-j}q^{(k-j)(k-j-1)/2}
	\left[\begin{array}{cc}
	k\\j
	\end{array}\right]_q
	f(q^jz_0)=0
$$
for any $z \in\C\setminus\{0\}$ and $0\leq k<n$, but $D_q^nf(z_0)\neq 0$. Therefore,
$z_0$ is a zero of $f(z)$ with $q$-weight $n$ in $G$.

If $z_0=0$, then $z_0$ is a zero of $f$ with $q$-weight $n$ if and only if $z_0$ is a zero of $f$ with multiplicity $n$. This completes the proof of the assertion. $\hfill\square$

\bigskip
\noindent

Obviously, the analytic function $f$ in \eqref{fzero.eq} reduces to $f(z)=(z-z_0)^ng(z)$, where $g$ is analytic in $z_0$ and $g(z_0)\neq 0$, as $q\to 1$.
From Theorem \ref{zero.theorem}, we have the following corollary.

\begin{corollary}\label{n-1.cor}
Let $f$ be analytic in the domain $G\subset\C$, and let $n\in\N^+$ and $q\in\C\setminus\{0\}$. Suppose that $z_0, qz_0,\ldots, q^nz_0\in G$.
If $z_0$ is an $a$-point of $f$ with $q$-weight $n$ in $G$, then $z_0$ is a zero of $D_q f$ with $q$-weight $n-1$.
\end{corollary}
\begin{proof}
It is straightforward to verify the conclusion when $z_0=0$.
Suppose that $z_0\neq 0$ is a zero of $f-a$ with $q$-weight $n$ in $G\subset\C$, where $a\in\C$. According to Theorem \ref{zero.theorem}, there exists an analytic function $h(z)$ in $G$, such that
$$
f(z)-a=[z-z_0]_q^nh(z),
$$
where $h(q^nz_0)\ne0$. Moreover,
\begin{align*}\nonumber
     D_qf(z) &= \dfrac{[qz-z_0]_q^n h(qz)-[z-z_0]_q^nh(z)}{qz-z} \\
             &= [z-z_0]_q^{n-1}\dfrac{q^{n-1}(qz-z_0)h(qz)-(z-q^{n-1}z_0)h(z)}{qz-z}\\
             &=: [z-z_0]_q^{n-1}H(z).
\end{align*}
It is obvious that $ H(q^{n-1}z_0)= [n]_qh(q^nz_0)\neq 0 $, which implies that $z_0$ is a zero of $D_qf(z)$ with $q$-weight $n-1$. Hence, we proved our assertion.
\end{proof}

\section{$Q$-difference radical}

Let $P$ be a polynomial with degree $p$, and $z_1$ be a zero of $P$ with $q$-weight $n_1$ and $P(q^{-1}z_1)\neq 0$, where $q\in\C\setminus\{0\}$. Theorem~\ref{zero.theorem} states that there exists a polynomial $P_1$ with degree $p-n_1$ such that $P(z)=[z-z_1]_q^{n_1}P_1(z)$, where $P_1(q^{n_1}z_1)\neq 0$.
Let $z_2$ be a zero of $P_1$ with $q$-weight $n_2$ and $P_1(q^{-1}z_2)\neq 0$.
Then there exists a polynomial $P_2$ with degree $p-n_1-n_2$ such that $P_1(z)=[z-z_2]_q^{n_2}P_2(z)$, where $P_2(q^{n_2}z_2)\neq 0$.
Repeating this argument for finitely many times, we see that $P(z)$ can be written uniquely as
\begin{equation}\label{P.eq}
P(z)=A\prod_{j=1}^N[z-z_j]^{n_j}_q,
\end{equation}
where $A$ is a nonzero constant, and $p=n_1+\cdots+n_N$.
Note that it is possible that $z_j=z_k$ even though $j\ne k$ in \eqref{P.eq}.
We define the \emph{$q$-difference radical ${\rm rad}_q(P)$} by product of these linear factors, i.e.,
\begin{equation}\label{q-radical}
{\rm rad}_q(P)=\prod_{j=1}^N(z-z_j).
\end{equation}
From \eqref{P.eq} and \eqref{q-radical}, we see that $\lim_{q\to 1}{\rm rad}_q(P)={\rm rad}(P)$.

Suppose that $f$ and $g$ are entire functions, and suppose that $z_1\in\C$ is a zero of $f$ with $q$-weight $m$ and $z_2\in\C$ is a zero of $g$ with $q$-weight $n$.
Then from Theorem~\ref{zero.theorem}, there exist entire functions $F$ and $G$ such that
$f(z)=[z-z_1]_q^{m_1}F(z)$ and $g(z)=[z-z_2]_q^{n_1}G(z)$, where $1\leq m_1\leq m$ and $1\leq n_1\leq n$.
Note that $F$ and $G$ are dependent on $m_1$ and $n_1$ respectively. If for some $m_1$, $n_1$,
    $$
    f(z)g(z)=[z-z_0]_q^{m_1+n_1}F(z)G(z),
    $$
where $z_0$ is $z_1$ or $z_2$, then $z-z_0$ is called the \emph{common $q$-divisor} of $f$ and $g$, which is the $q$-analogue of classical common divisor. If $z_0=z_1$, then $z_2=q^{m_1}z_1$, and if $z_0=z_2$, then $z_1=q^{n_1}z_2$.

\begin{example}
Suppose that $q\in\C$ such that $0<|q|<1$, and
$f(z)=(z-1)(z-q)(z-q^2)$ and $g(z)=(z-q^2)(z-q^3)(z-q^4)$. Then $z-1$, $z-q$, $z-q^2$ are the common $q$-divisor of $f$ and $g$.
In addition, $z-q^2$ is the common divisor of $f$ and $g$.
If we write $f=[z-1]_q^3$ and $g=(z-q^2)[z-q^3]_q^2$ or $f=[z-1]^2_q(z-q^2)$ and $g=[z-q^2]_q^3$, then
    $$
    fg=[z-1]_q^5(z-q^2).
    $$
It shows that $z-1$ is a common $q$-divisor of $f$ and $g$. Here $m_1$=3 and $n_1=2$ or $m_1=2$ and $n_1=3$.
If $h(z)=(z-q^3)(z-q^4)$, then $z-1$, $z-q$, $z-q^2$ are the also common $q$-divisors of $f$ and $h$. But there is no common divisor of $f$ and $h$.
\end{example}

 If $f$ and $g$ do not have any nonconstant common $q$-divisors, then $f$ and $g$ are called \emph{relatively $q$-prime}.
For functions $f_1$, $f_2$, $\dots$, $f_k$, $k>2$, we call $f_1$, $f_2$, $\dots$, $f_k$ are relatively $q$-prime if any pair of $f_i$ and $f_j$, $1\leq i, j\leq k$, $i\ne j$ are relatively $q$-prime.

Let $P$ and $Q$ be nonconstant polynomials. In general, we have
$$
\deg({\rm rad}_q (PQ))\leq
\deg({\rm rad}_q (P))+\deg({\rm rad}_q (Q)).
$$
 By the definition of the relatively $q$-prime, if $P$ and $Q$ are relatively $q$-prime, then
    \begin{equation}\label{PQ.eq}
    \deg({\rm rad}_q (PQ))=\deg({\rm rad}_q (P))+\deg({\rm rad}_q (Q)).
    \end{equation}
On the other hand, even though the equality \eqref{PQ.eq} holds, $P$ and $Q$ are not always relatively $q$-prime. For example, consider $P(z)=(z-1)$ and $Q(z)=(z-1)(z-q)$, where $q\in\C$ such that $0<|q|<1$. Then ${\rm rad}_q (PQ)=(z-1)^2$, and ${\rm rad}_q (P)=z-1$, ${\rm rad}_q (Q)=z-1$, which gives the equality in \eqref{PQ.eq} holds. However, $P$ and $Q$ are not relatively $q$-prime. In fact, we can write $P(z)=[z-1]_q$, $Q(z)=[z-1]_q[z-q]_q$ and $P(z)Q(z)=[z-1]_q^{1+1}\cdot[z-1]_q$.

In the following, we denote the greatest common divisor of $f$ and $g$ by $\text{gcd}(f,g)$.
It is well known that if two entire functions $f$ and $g$ are relatively prime, then $\text{gcd}(f,f')$ and $\text{gcd}(g,g')$ are also relatively prime. We state the $q$-analogue of this result as follows.

\begin{lemma}\label{q-prime}
Suppose that $q\in\C\setminus\{0\}$, and $f$ and $g$ are relatively $q$-prime entire functions. Then ${\rm gcd}(f,D_q f)$
and ${\rm gcd}(g,D_q g)$ are also relatively $q$-prime. In addition, ${\rm gcd}(f,D_q f)$
and ${\rm gcd}(g,D_q g)$ are relatively prime.
\end{lemma}

\begin{proof}
Assume that ${\rm gcd}(f,D_q f)$ and ${\rm gcd}(g,D_q g)$ have a common $q$-divisor.
There exist $z_1,z_2\in\C$, $m,n\in\N$ and entire functions $F,G$ such that
    \begin{equation}\label{fDf.eq}
    {\rm gcd}(f,D_q f)=[z-z_1]_q^mF(z),
    \end{equation}
and
    \begin{equation}\label{gDg.eq}
    {\rm gcd}(g,D_q g)=[z-z_2]_q^nG(z).
    \end{equation}
Without loss of generality, we suppose that $z_2=q^mz_1$. Then we have
    $$
    {\rm gcd}(f,D_q f){\rm gcd}(g,D_q g)=[z-z_1]^{m+n}_qF(z)G(z).
    $$
It follows from \eqref{fDf.eq}, \eqref{gDg.eq} and Corollary \ref{n-1.cor} that
    $$
    f(z)=[z-z_1]^{m+1}_qF_1(z),
    $$
where $F_1$ is an entire function, and
    $$
    g(z)=[z-z_2]^{n+1}_qG_1(z),
    $$
where $G_1$ is an entire function. Therefore, we have
    $$
    f(z)g(z)=[z-z_1]^{m+n+1}_q(z-q^mz_1)F_1(z)G_1(z).
    $$
It shows that $f$ and $g$ have a common $q$-divisor $z-z_1$, which is a contradiction
with our assumption. We thus proved the first assertion.

Suppose that ${\rm gcd}(f,D_q f)$ and ${\rm gcd}(g,D_q g)$ are not relatively prime.
We set $z-z_0$ is the common divisor of ${\rm gcd}(f,D_q f)$ and ${\rm gcd}(g,D_q g)$. Hence,
there exist entire functions $F_2, G_2$ such that
    $$
    {\rm gcd}(f,D_q f)=(z-z_0)F_2(z),
    $$
and
    $$
    {\rm gcd}(g,D_q g)=(z-z_0)G_2(z).
    $$
It yields by Corollary \ref{n-1.cor} that $f=[z-z_0]^2_qF_3$ and $g=[z-z_0]_q^2G_3$, where $F_3$ and $G_3$ are entire functions. It implies that $f$ and $g$ are not relatively shifting prime, which
is a contradiction to our assumption. We proved our Lemma \ref{q-prime}.
\end{proof}

Now let us state our main result in this paper as follows, which is $q$-difference analogue of the Stothers-Mason theorem.

\begin{theorem}\label{qabc.theorem}
Let $q\in\C\setminus\{0\}$ such that $|q|\neq 1$, and $a$, $b$ and $c$ be relatively $q$-prime polynomials such that
    $$
    a+b=c
    $$
and such that $a$, $b$ and $c$ are not all constants. Then
    \begin{equation}\label{diff-abc.eq}
    \max\{\deg(a),\deg(b),\deg(c)\}\leq \deg({{\rm rad}_q(abc)})-1.
    \end{equation}
\end{theorem}

\begin{proof}
Without loss of generality, let us assume that $\deg(c)=\max\{\deg(a),\deg(b),\deg(c)\}$.
Since $a+b=c$, we have $D_qa+D_qb=D_qc$. Multiplying
the first equation by $D_qa$, the second by $a$, and then subtracting them, we obtain
    \begin{equation}\label{ab.eq}
    bD_qa-aD_qb=cD_qa-aD_qc.
    \end{equation}
Suppose that $bD_qa-aD_qb\not\equiv 0$. It follows from \eqref{ab.eq} that
${\rm gcd}(a,D_qa)$, ${\rm gcd}(b,D_qb)$ and ${\rm gcd}(c,D_qc)$ all divide
$bD_qa-aD_qb$. Since $a,b,c$ are relatively $q$-prime polynomials, from Lemma \ref{q-prime}, ${\rm gcd}(a,D_qa)$, ${\rm gcd}(b,D_qb)$ and ${\rm gcd}(c,D_qc)$ are relatively $q$-prime polynomials, which are the factors of $bD_qa-aD_qb$. Therefore,
    $$
    \deg({\rm gcd}(a,D_qa))+\deg({\rm gcd}(b,D_qb))+\deg({\rm gcd}(c,D_qc))
    \leq \deg(a)+\deg(b)-1.
    $$
We add $\deg(c)$ to both sides, and use Corollary \ref{n-1.cor} and \eqref{PQ.eq},
    \begin{equation*}
    \begin{split}
    \deg(c)\leq& \big\{\deg(a)-\deg({\rm gcd}(a,D_qa))\big\}+\big\{\deg(b)-\deg({\rm gcd}(b,D_qb))\big\}\\
    &+\big\{\deg(c)-\deg({\rm gcd}(c,D_qc))\big\}-1\\
    =&\deg({\rm rad}_q(a))+\deg({\rm rad}_q(b))+\deg({\rm rad}_q(c))-1\\
    =&\deg({\rm rad}_q(abc))-1.
    \end{split}
    \end{equation*}
Therefore, we prove our assertion when $bD_qa-aD_qb\not\equiv 0$. In the following, we proceed to prove our theorem when $bD_qa-aD_qb\equiv 0$. Hence, we have
    $$
    \frac{D_qa}{a}=\frac{D_qb}{b}=\frac{D_qc}{c}=Q(z),
    $$
where $Q(z)$ is a rational function. Polynomials $a,b,c$ are solutions of linear $q$-difference equations
    $$
    f(qz)=(1+(q-1)zQ(z))f(z).
    $$
There exist $\pi_1,\pi_2$ such that $b=\pi_1a$ and $c=\pi_2a$, where $\pi_j(qz)=\pi_j(z)$ for $j=1,2$, see e.g., \cite{AAIM}. Since $b$ and $c$ are polynomials, the only possibility is $\pi_1$ and $\pi_2$ are constants. It yields that all zeros of $a,b,c$ are of simple $q$-weight. Hence,
    $$
    \deg(a)+\deg(b)+\deg(c)=\deg({\rm rad}_q(abc)).
    $$
It is easy to see that \eqref{diff-abc.eq} holds. We proved our assertion.
\end{proof}

Theorem \ref{qabc.theorem} can be seen as a generalization of Stothers-Mason theorem. As $q\to 1$,
Theorem \ref{qabc.theorem} is the classical Stothers-Mason theorem.
The following example shows that the assertion in Theorem \ref{qabc.theorem} is sharp.

\begin{example}
Let $a=[z-1]_q^2$, $b=-[z+1]_q^2$ and $c=-2(q+1)z$. We see that $a$, $b$
and $c$ are relatively $q$-prime polynomials satisfying $a+b=c$. Moreover,
it is shown that $\max\{\deg(a),\deg(b),\deg(c)\}=2$, ${{\rm rad}_q(abc)}=z(z^2-1)$,
and $\deg({{\rm rad}_q(abc)})=3$.
\end{example}

\section{Extension of the Stothers-Mason theorem with $q$-difference radical}

We write a polynomial $P$ in the form \eqref{P.eq}, i.e., $P(z)=A\prod_{j=1}^N[z-z_j]_q^{n_j}$.
The $q$-radical of truncation level $\mu\in\N$ for a polynomial $P(z)$ is denoted by
    \begin{equation}\label{radq.eq}
    {\rm rad}_q^\mu(P)={\rm gcd}\bigg(\prod_{j=1}^N[z-z_j]_q^{n_j}, \ \prod_{j=1}^N[z-z_j]_q^{\mu}\bigg).
    \end{equation}
Obviously, it is a generalization of $q$-difference radical. When $\mu=1$ in \eqref{radq.eq}, it reduces to the $q$-difference radical of $P$ given in \eqref{q-radical}.
We note that ${\rm rad}_q^\mu(P)={\rm rad}_q (P)$ holds for any $\mu\geq 2$ if and only if all the zeros of $P$ are of simple $q$-weight.
It follows from \eqref{radq.eq},
\begin{equation}\label{qmu.eq}
\deg({\rm rad}_q^\mu(P))=\deg\left(\prod_{j=1}^N[z-z_j]_q^{\min(n_j,\mu)}\right)=\sum_{j=1}^N \min(n_j,\mu).
\end{equation}
In addition, we have an inequality
    \begin{equation}\label{small.eq}
   \deg\left( {\rm rad}_q^\mu(P)\right)\leq \mu\cdot\deg \left({\rm rad}_q(P)\right).
    \end{equation}

The following theorem extends Theorem~\ref{qabc.theorem} for $m+1$ polynomials, where $m\in\N$, $m\geq2$. We state it as follows.
\begin{theorem}\label{fm.theorem}
Let $m\in \N$, $m\geq2$ and let $f_1,\ldots,f_{m+1}$  be pairwise relatively $q$-prime polynomials with
    $
    \min_{1\leq i\leq m+1} \{\deg f_i\}\geq m-1
    $
satisfying the following functional equation
    $$
    f_1+ \cdots + f_m=f_{m+1},
    $$
and such that $f_1,\dots,f_{m}$ are linearly independent over $\C$. Then
    \begin{equation}\label{q-diffMasonineq}
    \begin{split}
    \max_{1\leq i\leq m+1}\{\deg f_i\}&\leq \deg({\rm rad}_q^{m-1}(f_1f_2\cdots f_{m+1}))-\frac{1}{2}m(m-1)\\
    &\leq (m-1)\deg({\rm rad} _q(f_1f_2\cdots f_{m+1}))-\frac{1}{2}m(m-1).
    \end{split}
    \end{equation}
\end{theorem}

In order to prove Theorem \ref{fm.theorem}, we need some lemmas below. Let $[x]^{+}=\max\{x,0\}$ for $x\in \mathbb{R}$.

\begin{lemma}\label{gcd.lemma}
Let $n\in \mathbb{N}$ and let $P$ be a polynomial with $\deg P \ge n$ represented by \eqref{P.eq}. Then
    \begin{equation}\label{deg.gcd1}
    \deg(\gcd(P,D_qP,...,D_q^nP))=\sum_{j=1}^N [n_j-n]^{+}
    \end{equation}
and
    $$
    \deg(P)-\deg(\gcd(P,D_qP,...,D_q^nP))=\deg({\rm rad} _q^{n}(P)).
    $$
\end{lemma}

\begin{proof}
Let $z_j$ be a zero of $P$ with $q$-weight $n_j$ such that $P(q^{-1}z_j)\ne 0$. From Theorem \ref{zero.theorem}, we write $P(z)=[z-z_j]_q^{n_j}H_j(z)$, where $H_j(z)$ is a polynomial with $H_j(q^{n_j}z_j)\ne 0$ for some $j\in\N$. By Corollary \ref{n-1.cor}, $z_j$ is a zero of $D_qP$ with $q$-weight $n_j-1$, that is $D_qP(z) = [z-z_j]_q^{n_j-1}H_{j,1}(z)$, where $H_{j,1}(z)$ is a polynomial with $H_{j,1}(q^{n_j-1}z_j)\ne 0$.

Let $k\in \mathbb{N}$, such that $1\le k \le n$. If $n_j \ge n+1$, $D_q^kP(z)=[z-z_j]_q^{n_j-k}H_{j,k}(z)$, where $H_{j,k}(z)$ is a polynomial satisfying $H_{j,k}(q^{n_j-k}z_j)\ne 0$, which implies ${\gcd(P,D_qP,...,D_q^nP)}$ is divided by $[z-z_j]_q^{n_j-n}$.
If $n_j\le n$, the factors in $[z-z_j]_q^{n_j}$ never contribute to ${\gcd(P,D_qP,...,D_q^nP)}$. From the argument above, we have
$$
    \gcd(P,D_qP,...,D_q^nP)=\prod_{j=1}^{N}[z-z_j]_q^{[n_j-n]^{+}},
$$
which implies (\ref{deg.gcd1}). By (\ref{deg.gcd1}) and (\ref{qmu.eq}), we obtain
\begin{align}\nonumber
&\deg(\gcd(P,D_qP,...,D_q^nP))=\sum_{j=1}^N [n_j-n]^{+}=\sum_{j=1}^{N}\max\{n_j-n,0\}
\\\nonumber
&\quad = \sum_{j=1}^{N}n_j - \sum_{j=1}^{N}\min\{n_j,n\} = \deg(P)-\deg({\rm rad}_q^n(P)).
\end{align}
This completes the proof of Lemma \ref{gcd.lemma}.
\end{proof}

For any nonconstant polynomials $P$ and $Q$, we have the following estimate
\begin{equation}\label{degrad}
    \deg ({\rm rad}_q^n(PQ) ) \le \deg ({\rm rad}_q^n(P) )+\deg ({\rm rad}_q^n(Q) ),
\end{equation}
where $n \in \mathbb{N}$. Moreover, if $P$ and $Q$ are relatively $q$-prime, then
    $$
     \deg ({\rm rad}_q^n(PQ) ) = \deg ({\rm rad}_q^n(P) )+\deg ({\rm rad}_q^n(Q) ).
    $$
Let now $m \in \mathbb{N}$, and let $f_j(z), j=1,2,...,m$ be polynomials. Let us recall $q$-Casorati determinant of entire functions $f_1,f_2,...,f_m$ by, see e.g., \cite[Pages 4282-4283]{HKT2014} or \cite{AAIM},
\begin{equation}\label{def.q-Wronskain}
\begin{split}
    W_q(z)=W_q(f_1,f_2,...,f_m)&=
    \begin{vmatrix}
    f_1(z) & f_2(z) & \cdots & f_m(z)\\
    D_qf_1(z) & D_qf_2(z) & \cdots &D_qf_m(z) \\
    \vdots & \vdots & \ddots & \vdots \\
    D_q^{m-1}f_1(z) & D_q^{m-1}f_2(z) & \cdots &D_q^{m-1}f_m(z)
    \end{vmatrix}\\[7pt]
    &= \frac{
    \begin{vmatrix}
    f_1(z) & f_2(z) & \cdots & f_m(z)\\
    f_1(qz) & f_2(qz) & \cdots & f_m(qz) \\
    \vdots & \vdots & \ddots & \vdots \\
    f_1(q^{m-1}z) & f_2(q^{m-1}z) & \cdots &f_m(q^{m-1}z)
    \end{vmatrix}}{(qz-z)^{m-1}}.
\end{split}
\end{equation}
The $q$-Casorati determinant is a useful tool for analyzing the solutions of $q$-difference equations, see~\cite{Wen2014}. By the properties of the determinant, we have
\begin{lemma}\label{lm.q-Wronskain}
If $q\in\C\setminus\{0,1\}$, then the $q$-Casorati determinant vanishes identically on $\C$ if and only if the functions $f_1,\ldots,f_m$ are linearly dependent over the field of functions
$\pi(z)$ satisfying $\pi(qz)=\pi(z)$.
\end{lemma}

\noindent
\emph{Proof of Theorem {\rm \ref{fm.theorem}}}.
Let $W_q(z)$ be the $q$-Casorati determinant of polynomials $f_1,f_2,...,f_m$ defined by \eqref{def.q-Wronskain}. It yields that $W_q(z) \not\equiv 0$ by
Lemma \ref{lm.q-Wronskain} and the fact $f_1,f_2,...,f_m$ are linearly independent over $\C$.
Now let us write
    $$
    f_j(z) = A_j \prod_{k=1}^{N_j}[z-z_{j,k}]_q^{n_{j,k}},
    $$
where $A_j$ are constants, $N_j,n_{j,k} \in \N^+$ and $f_j(z_{j,k}q^{-1}) \ne 0$ for
 $j=1,\ldots, m+1$. It implies that $f_j(z)$ is divided by $\gcd(f_j,D_qf_j,...,D_q^{m-1}f_j)=\prod_{k=1}^{N_j}[z-z_{j,k}]_q^{[n_{j,k}-(m-1)]^+}$ for any $j=1,\ldots, m+1$. Obviously, every $f_j(z)~(j=1,2,...,m)$ can be replaced by $f_{m+1}=f_1+f_2+\cdots+f_m$ in $W_q(z)$ by the properties of
determinant.
Moreover, $W_q(z)$ is divided by $\gcd (f_j,D_qf_j,...,D_q^{m-1}f_j)$ for $j \in \{ 1,2,...,m+1\}$, and $f_1,\ldots,f_{m+1}$ are pairwise relatively $q$-prime polynomials, it gives that $W_q(z)$ is divided by
    $$
    Q(z):=\prod_{j=1}^{m+1}\prod_{k=1}^{N_j}[z-z_{j,k}]_q^{[n_{j,k}-m+1]^+},
    $$
which implies that there exists a polynomial $P(z)$ such that $W_q(z)=P(z)Q(z)$.

By the assumption $\min_{1\leq j\leq m+1} \{\deg f_j\}\geq m-1$, we have $\deg (D_q^n f_j)=\deg f_j-n \ge 0$ with $0 \le n \le m-1$. This implies that $\deg W_q(z)$ is never greater than any sum of distinct $m$ of the $\deg f_j(z), 1\le j \le m+1$, minus $\sum_{n=0}^{m-1}n=m(m-1)/2$ as the sum of $\deg (D_q^n f_{j_v})$ for the mutually distinct $m$ integers $j_v \in \{ 1,2,...,m+1 \}$. Since $m \in \mathbb{N}$ and $f_1,f_2,...,f_{m+1}$ are pairwise relatively $q$-prime, by Lemma \ref{gcd.lemma} and \eqref{degrad}, we have
\begin{align}\nonumber
&\min_{1\leq h\leq m+1} \bigg\{\sum_{1 \le j\le m+1,j\ne h} \deg f_j \bigg\}-\frac{1}{2}m(m-1)
\ge \deg \left(\prod_{j=1}^{m+1} \gcd (f_j,D_qf_j,...,D_q^{m-1}f_j )\right)\\\nonumber
=& \sum_{j=1}^{m+1}(\deg (f_j)-\deg ({\rm rad}_q^{m-1} (f_j)))
= \sum_{j=1}^{m+1} \deg (f_j)- \deg ({\rm rad}_q^{m-1}(f_1f_2\cdots f_{m+1})),
\end{align}
which implies the first inequality of \eqref{q-diffMasonineq}. Combining \eqref{small.eq} and the first inequality of \eqref{q-diffMasonineq}, we have the second inequality of \eqref{q-diffMasonineq} follows. We have proved Theorem \ref{fm.theorem}.  $\hfill\square$

\bigskip
In the following, let us state an example to show the inequality in \eqref{q-diffMasonineq} is sharp.
\begin{example}
Let $0<|q|<1$. Set $f_1=[z-1]_q^5$, $f_2=-[z+1]_q^5$, $f_3=2[5]_qz^4$. It is obvious that $f_1$,$f_2$ and $f_3$ are linearly independent. We define
    $$
    f_4=f_1+f_2+f_3=   -2q^3 \left[\begin{array}{cc}
    5\\
    2
    \end{array}
    \right]_qz^2-2q^{10}.
    $$
Let us choose suitable $q$ such that $\pm q^{-1}, 0,\pm q,\pm q^2,\pm q^3,\pm q^4, \pm q^5$ are not zeros of $f_4$. It yields that $f_1$, $f_2$, $f_3$ and $f_4$ are pairwise relatively $q$-prime.
According to \eqref{radq.eq}, we have $\deg({\rm rad}_q^2(f_1))=2$, $\deg({\rm rad}_q^2(f_2))=2$, $\deg({\rm rad}_q^2(f_3))=2$, $\deg({\rm rad}_q^2(f_4))=2$ and $\max_{1\le i\le4}\{\deg f_i\}=5$. It shows that the inequality in \eqref{q-diffMasonineq} is sharp.

\end{example}

\section{Polynomial solutions of $q$-difference Fermat functional equations}

In this section, we apply Theorem~\ref{qabc.theorem} and Theorem~\ref{fm.theorem} to $q$-difference Fermat type functional equations for investigating nonexistence of polynomial solutions.
We adopt the expression $[P]_q^n:=P(z)P(qz)\cdots P(q^{n-1}z)$ instead of $P^n$ for a polynomial $P$ in the Fermat type functional equations.  For the Fermat type functional equations, see e.g., \cite{GH2004},~\cite{GIK}.

\begin{theorem}\label{Equ3}
Let $a$, $b$ and $c$ be polynomials, not all constants, $q\in\C\setminus\{0\}$ satisfying $|q|\neq 1$, and $n\in\N$ such that $[a]_q^n$, $[b]_q^n$ and $[c]_q^n$ are relatively $q$-prime and satisfy
    \begin{equation}\label{abc.equ}
    [a]_q^n+[b]_q^n=[c]_q^n.
    \end{equation}
Then $n\leq 2$. In addition, if one of $a$, $b$ and $c$ is a constant, then $n=1$.
\end{theorem}

\begin{proof}
Let us first assume that $a$, $b$ and $c$ are all nonconstant polynomials.
By Theorem~\ref{qabc.theorem}, we have
    \begin{equation}\label{a.eq}
    \begin{split}
    n\deg(a)=\deg([a]_q^n)&\leq\max\{ \deg([a]_q^n), \deg([b]_q^n),\deg([c]_q^n)\}\\
    &\leq \deg({\rm rad}_q([a]_q^n~[b]_q^n~[c]_q^n))-1\\
    &=\deg({\rm rad}_q([a]_q^n))+
    \deg({\rm rad}_q([b]_q^n))+
    \deg({\rm rad}_q([c]_q^n))-1\\
    &\leq \deg(a)+\deg(b)+\deg(c)-1.
    \end{split}
    \end{equation}
Similarly, for $b$ and $c$, we have
    \begin{equation}\label{b.eq}
    n\deg(b)\leq \deg(a)+\deg(b)+\deg(c)-1
    \end{equation}
and
    \begin{equation}\label{c.eq}
    n\deg(c)\leq \deg(a)+\deg(b)+\deg(c)-1.
    \end{equation}
Combining inequalities \eqref{a.eq}, \eqref{b.eq} and \eqref{c.eq}, we obtain
    $$
    n(\deg(a)+\deg(b)+\deg(c))\leq 3(\deg(a)+\deg(b)+\deg(c))-3,
    $$
which implies that $n\leq 2$.

We secondly assume one of $a$, $b$ and $c$ is a constant. Without loss of generality, we assume that $c$ is a constant. Then \eqref{a.eq} and \eqref{b.eq} yield
    $$
    n(\deg(a)+\deg(b))\leq 2(\deg(a)+\deg(b))-2,
    $$
which shows that $n\leq 1$. We proved our assertion.
\end{proof}

\begin{theorem}\label{Equn.eq}
Let $m\in \N$, $m\geq 2$, $n\in\N$ and $q\in\C\setminus\{0\}$ such that $|q|\neq 1$. Suppose that there exist $f_1,\ldots,f_{m+1}$ nonconstant polynomials satisfying
    $$
    [f_1]_q^n+[f_2]_q^n+\cdots
    +[f_m]_q^n=[f_{m+1}]_q^n.
    $$
Further suppose  that $[f_1]_q^n$, $\ldots$, $[f_{m+1}]_q^n$ are pairwise relatively $q$-prime and $[f_1]_q^n,\ldots,[f_{m+1}]_q^n$ are linearly independent.
Then
$$
n\leq m^2-1-\frac{m(m-1)}{2\max_{1\leq i\leq m+1}\{\deg(f_i)\}}.
$$
\end{theorem}

\begin{proof}
By using Theorem~\ref{fm.theorem}, we have
    \begin{equation*}
    \begin{split}
    n \max_{1\leq j\leq m+1}\{\deg(f_j)\}
    &\leq (m-1)\deg\left(\text{rad}_q\left([f_1]_q^n[f_2]_q^n\cdots [f_{m+1}]_q^n\right)\right)-\frac12m(m-1)\\
    &\leq (m+1)(m-1)\max_{1\leq j\leq m+1}\{\deg(f_j)\}-\frac12m(m-1).
    \end{split}
    \end{equation*}
Since $\max_{1\leq j\leq m+1}\{\deg(f_j)\}\geq 1$, we proved our assertion.

\end{proof}

\bigskip

\noindent
\emph{J.-T. Lu}\\
\textsc{Shantou University, Department of Mathematics,\\
Daxue Road No.~243, Shantou 515063, China}\\
\texttt{Email:24jtlu@stu.edu.cn}

\bigskip
\noindent
\emph{X.-X. Lu}\\
\textsc{Ankang University, School of Mathematics and Statistics,\\
Yucai Road No.~92, Ankang 725000, China}\\
\texttt{Email:luxx@aku.edu.cn}

\bigskip
\noindent
\emph{Z.-T.~Wen}\\
\textsc{Shantou University, Department of Mathematics,\\
Daxue Road No.~243, Shantou 515063, China}\\
\texttt{Email:zhtwen@stu.edu.cn}

\vspace{1cm}
\end{document}